\documentclass[11pt]{article}
\usepackage{amsmath}
\usepackage{amsthm}
\usepackage{amssymb}
\usepackage{graphicx}
\usepackage{multirow} 
\usepackage{tabularx}
\newcommand{\Z}{\mathbb{Z}}
\newcolumntype{R}{>{\raggedleft\arraybackslash}X}
\newcolumntype{L}{>{\raggedright\arraybackslash}X}
\usepackage[a4paper,margin=3cm,right=3cm,includehead]{geometry}
\newtheorem{theorem}{Theorem}

\title{ \bf Distance partitions of extremal and largest known circulant graphs of degree 2 to 9}
\author{R. R. Lewis\\[-3pt]
\small Department of Mathematics and Statistics\\[-3pt]
\small The Open University\\[-3pt]
\small Milton Keynes, UK\\[-3pt]
\small \texttt{robert.lewis@open.ac.uk}}
\date{5th August 2014}



\makeatletter
\renewcommand\section{\@startsection {section}{1}{\z@}%
                                   {-2.5ex \@plus -1ex \@minus -.2ex}%
                                   {1.3ex \@plus.2ex}%
                                   {\normalfont\bf}}
\makeatother

\begin{document}
\maketitle

\begin{abstract}

This paper considers the degree-diameter problem for extremal and largest known undirected circulant graphs of degree 2 to 9 of arbitrary diameter. As these graphs are vertex transitive it is possible to define their distance partition. The number of vertices in each level of the distance partition is shown to be related to an established upper bound for the order of Abelian Cayley graphs. Furthermore these graphs are all found to have odd girth which is maximal for their diameter. Therefore the type of each vertex in a level may be well-defined by the number of adjacent vertices in the preceding level. With this definition the number of vertices of each type in each level is also shown to be related to the same Abelian Cayley graph upper bound. Finally some implications are discussed for circulant graphs of higher degree.

\end{abstract}


\section{Introduction}
The degree-diameter problem is to identify extremal graphs, having the largest possible number of vertices for a given maximum degree and diameter. For undirected circulant graphs of given degree and arbitrary diameter, extremal graphs have been identified only for degree 2 to 5. For degree 2 and 3, the solutions are straightforward. For degree 4 Chen and Jia included a proof in their 1993 paper \cite {Chen}, and for degree 5 Dougherty and Faber presented a proof in 2004 \cite {Dougherty}. For degree 6 to 9 families of largest known circulant graphs have been discovered which are conjectured to be extremal but have been proven so only for a limited range of diameters: degree 6 and 7 by Dougherty and Faber \cite {Dougherty} and degree 8 and 9 by the author \cite {Lewis}. For degree 10 and above a lower bound has been established for all even degrees by Chen and Jia \cite {Chen}.

Adopting the terminology of Macbeth, Siagiova and Siran \cite {Macbeth} we will use $d$ for degree, $k$ for diameter, and $CC(d,k)$ for the order of an extremal undirected circulant graph of degree $d$ and diameter $k$. We also use $CJ(d,k)$ for Chen and Jia's lower bound, $DF(d,k)$ for the order of Dougherty and Faber's graphs of degree 6 and 7, and $L(d,k)$ for the order of the author's graphs of degree 8 and 9. Formulae, depending on the diameter, for the order of the extremal and largest known circulant graphs of degree 2 to 9 are shown in table \ref{table:2A}.

\begin{table} [!htbp]
\small
\caption{\small The order and number of isomorphism classes of the extremal and largest known circulant graphs of degree 2 to 9 of arbitrary diameter $k$.} 
\centering 
\begin{tabular}{ @ { } c l l l c c c } 
\noalign {\vskip 2mm}  
\hline\hline 
\noalign {\vskip 1mm}  
Degree, & Status & Validity & Order &  \multicolumn {3} {r} { Isomorphism}\\
$d$ & & & & \multicolumn {3} {r} {classes} \\
\hline
\noalign {\vskip 1mm}  
2 & Extremal & $k\geq 1$ & \multicolumn {2} {l}{$CC(2,k)=2k+1$} & 1 & \\
3 & Extremal & $k\geq 1$ & \multicolumn {2} {l}{$CC(3,k)=4k$} & 1 & \\
\noalign {\vskip 1mm}  
4 & Extremal & $k\geq 1$ & \multicolumn {2} {l}{$CC(4,k)=2k^2+2k+1$} & 1 & \\
5 & Extremal & $k\geq 2$ & \multicolumn {2} {l}{$CC(5,k)=4k^2$} & 1 & \\
\noalign {\vskip 2mm}  
6 & Largest & $k\equiv0$ (mod 3), $k\geq 3$ & \multicolumn {2} {l}{$DF(6,k)=(32k^3+48k^2+54k+27)/27$} & 2 & \\
   & known & $k\equiv1$ (mod 3), $k\geq 1$ & \multicolumn {2} {l}{$DF(6,k)=(32k^3+48k^2+78k+31)/27$} & 1 & \\
   &            & $k\equiv2$ (mod 3), $k\geq 2$ & \multicolumn {2} {l}{$DF(6,k)=(32k^3+48k^2+54k+11)/27$} & 2* & \\
\noalign {\vskip 2mm}  
7 & Largest & $k\equiv0$ (mod 3), $k\geq 3$ & \multicolumn {2} {l}{$DF(7,k)=(64k^3+108k)/27$} & 1 & \\
   &  known & $k\equiv1$ (mod 3), $k\geq 4$ & \multicolumn {2} {l}{$DF(7,k)=(64k^3+60k-16)/27$} & 2 & \\
   &             & $k\equiv2$ (mod 3), $k\geq 5$ & \multicolumn {2} {l}{$DF(7,k)=(64k^3+60k+16)/27$} & 2 & \\
\noalign {\vskip 2mm}  
8 & Largest & $k\equiv0$ (mod 2), $k\geq 4$ & \multicolumn {2} {l}{$L(8,k)=(k^4+2k^3+6k^2+4k)/2$} & 1 & \\
   & known  & $k\equiv1$ (mod 2), $k\geq 3$ & \multicolumn {2} {l}{$L(8,k)=(k^4+2k^3+6k^2+6k+1)/2$} & 1 & \\
\noalign {\vskip 2mm}  
9 & Largest & $k\equiv0$ (mod 2), $k\geq 6$ & \multicolumn {2} {l}{$L(9,k)=k^4+3k^2+2k$} & 1 & \\
   & known  & $k\equiv1$ (mod 2), $k\geq 5$ & \multicolumn {2} {l}{$L(9,k)=k^4+3k^2$} & 2 & \\

\hline
\noalign {\vskip 1mm}
  & & & \multicolumn {4} {l} {* For $k=2$ there are three isomorphism classes}
\end{tabular}
\label{table:2A} 
\end{table}

An undirected circulant graph $X(\Z_n,C)$ of order $n$ may be defined as a Cayley graph whose vertices are the elements of the cyclic group $\Z_n$ where two vertices $i,j$ are connected by an arc $(i, j)$ if and only if $j-i$ is an element of $C$, an inverse-closed subset of $\Z_n \setminus 0$, called the connection set. In common with all Cayley graphs, circulant graphs are vertex transitive and regular, with the degree $d$ of each vertex equal to the size of $C$. If $n$ is odd then $\Z _n \setminus 0 $ has no elements of order $2$. Therefore $C$ has even size, say $d=2f$, and is comprised of $f$ complementary pairs of elements with one of each pair strictly between $0$ and $n/2$. The set of $f$ elements of $C$ between 0 and $n/2$ is defined to be the generator set $G$ for $X$. If $n$ is even then $\Z_n \setminus 0$ has just one element of order $2$, namely $n/2$. In this case $C$ is comprised of $f$ complementary pairs of elements, as for odd $n$, with or without the addition of the self-inverse element $n/2$. If $C$ has odd size, so that $d=2f+1$, then the value of the self-inverse element $n/2$ is fixed by $n$. Therefore for a circulant graph of given order and degree, its connection set $C$ is completely defined by specifying its generator set $G$. The size of the connection set is equal to the degree $d$ of the graph, and the size of the generator set $f$ is defined to be the dimension of the graph.
In summary, undirected circulant graphs of odd degree $d$ must have even order. They have dimension $f=(d-1)/2$. Graphs of even degree $d$ may have odd or even order. They have dimension $f=d/2$.


\section{Distance partitions of extremal and best circulant graphs}

As Cayley graphs are vertex transitive it is possible to define their distance partition. If $n_l$ denotes the number of vertices in the $l$th level of the partition with respect to any reference vertex, then the \it distance partition profile \rm of any vertex transitive graph is defined to be the $(k+1)$-vector $(1,n_1,n_2,\ldots,n_k)$ where $k$ is the diameter of the graph.
For extremal circulant graphs of one dimension we see easily that for degree 2 the distance partition profiles for increasing diameter $k \geq 1$ are $(1,2),(1,2,2), (1,2,2,2), (1,2,2,2,2),\ldots$ ; and for degree 3 are $(1,3), (1,3,4), (1,3,4,4), (1,3,4,4,4),\ldots$ . In both cases the size of each successive level is a constant (2 or 4 respectively). For extremal circulant graphs of two dimensions we find for degree 4 the sequence $(1,4), (1,4,8),$ $(1,4,8,12), (1,4,8,12,16),$ etc, so that each successive level is 4 more than the previous. For degree 5, each new level does not immediately take its final size. From Level 2 onwards the size of each new level is 2 below its maximum value, which it reaches when the next level is added. See Table \ref{table:4A}.

\begin{table} [h]
\footnotesize
\caption{\small Distance partition profiles for extremal circulant graphs of degree 5.} 
\centering 
\begin{tabular}{ @ { } c r *{7} {r} }
\noalign {\vskip 2mm}
\hline\hline 
Diameter & Order & \multicolumn {7} {l} {Distance partition level} \\ 
$k$ & $CC(5,k)$ & 0 & 1 & 2 & 3 & 4 & 5 & 6 \\
\hline 
1 & 6 & 1 & 5 \\
2 & 16 & 1 & 5 & 10 \\
3 & 36 & 1 & 5 & 12 & 18 \\
4 & 64 & 1 & 5 & 12 & 20 & 26 \\
5 & 100 & 1 & 5 & 12 & 20 & 28 & 34 \\
6 & 144 & 1 & 5 & 12 & 20 & 28 & 36 & 42 \\ 
\hline
\end{tabular}
\label{table:4A} 
\end{table}

From level 3 onwards the maximum size of each successive level is 8 more than the previous. So for both degree 3 and 4 the size of the levels increase at a constant rate.

For three dimensions the evolution of the distance partition profiles for increasing diameter $k$ becomes more complicated. The profiles for the largest known graphs of degree 6 for diameter $k \leq 15$ are shown in Table \ref{table:4B}. For diameters with two isomorphism classes both have the same distance partition profile.

\begin{table} [h]
\footnotesize
\caption{\small Distance partition profiles for largest known circulant graphs of degree 6.} 
\centering 
\begin{tabularx} {\linewidth} { @ { } c  @ {} r *{16} {R} }
\noalign {\vskip 1mm}  
\hline\hline 
Diameter & Order & \multicolumn {16} {l} {Distance partition level} \\ 
$k$ & $DF(6,k)$ & 0 & 1 & 2 & 3 & 4 & 5 & 6 & 7 & 8 & 9 & 10 & 11 & 12 & 13 & 14  & 15\\
\hline 
    1     & 7     & 1     & 6     &       &       &       &       &       &       &       &       &       &       &       &       &         \\
    2     & 21    & 1     & 6     & 14    &       &       &       &       &       &       &       &       &       &       &       &         \\
    3     & 55    & 1     & 6     & 18    & 30    &       &       &       &       &       &       &       &       &       &       &         \\
    4     & 117   & 1     & 6     & 18    & 38    & 54    &       &       &       &       &       &       &       &       &       &         \\
    5     & 203   & 1     & 6     & 18    & 38    & 62    & 78    &       &       &       &       &       &       &       &       &         \\
    6     & 333   & 1     & 6     & 18    & 38    & 66    & 94    & 110   &       &       &       &       &       &       &       &         \\
    7     & 515   & 1     & 6     & 18    & 38    & 66    & 102   & 134   & 150   &       &       &       &       &       &       &         \\
    8     & 737   & 1     & 6     & 18    & 38    & 66    & 102   & 142   & 174   & 190   &       &       &       &       &       &         \\
    9     & 1027  & 1     & 6     & 18    & 38    & 66    & 102   & 146   & 190   & 222   & 238   &       &       &       &       &         \\
    10    & 1393  & 1     & 6     & 18    & 38    & 66    & 102   & 146   & 198   & 246   & 278   & 294   &       &       &       &         \\
    11    & 1815  & 1     & 6     & 18    & 38    & 66    & 102   & 146   & 198   & 254   & 302   & 334   & 350   &       &       &         \\
    12    & 2329  & 1     & 6     & 18    & 38    & 66    & 102   & 146   & 198   & 258   & 318   & 366   & 398   & 414   &       &         \\
    13    & 2943  & 1     & 6     & 18    & 38    & 66    & 102   & 146   & 198   & 258   & 326   & 390   & 438   & 470   & 486   &         \\
    14    & 3629  & 1     & 6     & 18    & 38    & 66    & 102   & 146   & 198   & 258   & 326   & 398   & 462   & 510   & 542   & 558     \\
    15    & 4431  & 1     & 6     & 18    & 38    & 66    & 102   & 146   & 198   & 258   & 326   & 402   & 478   & 542   & 590   & 622   & 638 \\
\hline
\end{tabularx}
\label{table:4B} 
\end{table}

As for degree 5, the size of each new level is initially below its maximum value, but now the number of increments to reach its maximum is not fixed at 1 but increases with increasing diameter, so that the maximal zone, where the levels have reached their maximum value, covers about the first two thirds of the levels. Also the difference between the maximum size of successive levels is no longer constant but increases linearly. Degree 7 is similar, with distance partition profile independent of isomorphism class and maximal zone covering two thirds of the levels.

For the largest known circulant graphs of dimension 4 the evolution of the distance partition profiles follows a similar structure but with certain differences. For both degree 8 and 9 the maximal zone covers about the first half of the levels, and the difference between the maximum size of successive levels increases as a quadratic. For degree 9 the two isomorphism classes for odd diameter share the same profile.

The use of distance partitions in the analysis of extremal graphs can be taken a stage further by considering an extension of the intersection array. We define the \it total intersection array \rm of a vertex transitive graph to have the same format as a standard intersection array but where each element is the total number of adjacent vertices summed across all vertices within each level of the distance partition. So taking Godsil and Royle's example of the dodecahedron \cite{Godsil} which has distance partition profile $(1,3,6,6,3,1)$ and standard intersection array
\[ \left (
\begin{array} {l l l l l l}
-  & 1 & 1 & 1 & 2 & 3 \\
0 & 0 & 1 & 1 & 0 & 0 \\
3 & 2 & 1 & 1 & 1 & -
\end{array}
\right )
,\]
its total intersection array becomes 
\[ \left (
\begin{array} {l l l l l l}
-  & 3 & 6 & 6 & 6 & 3 \\
0 & 0 & 6 & 6 & 0 & 0 \\
3 & 6 & 6 & 6 & 3 & -
\end{array}
\right )
.\]

With this definition the sum of the elements in each column of the total intersection array is equal to the corresponding element of the distance partition profile multiplied by the degree. The advantage for our analysis of circulant graphs is that the total intersection array is defined for all vertex transitive graphs whereas standard intersection arrays are only defined for distance regular graphs. Total intersection arrays provide a useful view on the structure of the graphs. The first non-zero element in the middle row determines the odd girth of the graph. They can also distinguish non-isomorphic graphs of common degree, diameter and order that might have the same distance partition profile. An example is given by the four isomorphism classses of extremal circulant graphs of degree 9 and diameter 3. These are easily proven to be distinct by determining their total intersection arrays, which are all different; see Table \ref{table:4F}.

\begin{table} [!htbp]
\small
\caption{\small Extremal graphs of each of the four isomorphism classes for $d=9$ and $k=3$.} 
\centering 
\begin{tabular}{ l l l l l}
\noalign {\vskip 1mm} 
\hline \hline
\noalign {\vskip 1mm} 
Diameter, & Order, & Generator set & Distance & Total \\
$k$ & $n$ &(excluding $n/2$) & partition & intersection \\
& & & profile & array \\
\hline
\noalign {\vskip 2mm} 
3 & 130 & $\left \{ 1,8,14,47 \right \}$ & $ (1,9,40,80)$ & 
 \multirow {3} {*} {$\left ( \begin{array} {l l l l}
-  & 9 & 72 & 244 \\
0 & 0 & 44 & 476  \\
9 & 72 &  244 & -
\end{array}
\right )$} \\
& & & & \\
& & & & \\
& & & & \\
3 & 130 & $\left \{ 1,8,20,35 \right \}$ & $(1,9,40,80)$ & 
 \multirow {3} {*} {$\left ( \begin{array} {l l l l}
-  & 9 & 72 & 242 \\
0 & 0 & 46 & 478  \\
9 & 72 &  242 & -
\end{array}
\right )$} \\
& & & & \\
& & & & \\
& & & & \\
3 & 130 & $\left \{ 1,26,49, 61 \right \}$ & $(1,9,40,80)$ & 
 \multirow {3} {*} {$\left ( \begin{array} {l l l l}
-  & 9 & 72 & 286 \\
0 & 0 & 2 & 434  \\
9 & 72 &  286 & -
\end{array}
\right )$} \\
& & & & \\
& & & & \\
& & & & \\
3 & 130 & $\left \{ 2,8,13,32 \right \}$ & $(1,9,40,80)$ & 
 \multirow {3} {*} {$\left ( \begin{array} {l l l l}
-  & 9 & 72 & 234 \\
0 & 0 & 54 & 486  \\
9 & 72 &  234 & -
\end{array}
\right )$} \\
& & & & \\
& & & & \\
\noalign {\vskip 1mm}
\hline
\end{tabular}
\label{table:4F} 
\end{table}

Reverting to the discussion on distance partition profiles, some obvious questions arise. What is the structure behind the maximum size of each level? What is the logic behind the evolution of the size of each level until it reaches its maximum? What determines the number of levels in the maximal zone?

\newpage
\section{Distance partitions and maximal levels}

First we consider the maximum size of each level.
A useful upper bound for the order of any Abelian Cayley graph is given by Doughterty and Faber \cite{Dougherty}. For a graph of degree $d$ and diameter $k$ the upper bound $M_{AC}(d,k)$  is defined as follows: 
\[M_{AC}(d,k) =
\begin{cases}
\ S(f,k) &\mbox{ for even } d, \mbox{ where } f=d/2\\
\ S(f,k)+S(f,k-1) &\mbox { for odd } d, \mbox{ where } f=(d-1)/2
\end{cases}
\]
where $f$ is the dimension of the graph and $S(f,k)= \sum _{i=0}^f 2^i {f\choose i} {k\choose i}$.

If we consider the implied upper bound for the size of level $l$ determined by the first order difference of the upper bound $M_{AC}(d,l)$ for increasing $l$, so that $LM_{AC}(d,l)=M_{AC}(d,l)-M_{AC}(d,l-1)$ for degree $d \geq 2$ and level $l \geq 2$, then the result is shown in Table \ref{table:4D}.

\begin{table} [h]
\small
\caption{\small $LM_{AC}(d,l)$, first order difference of the upper bound, $M_{AC}(d,l)$.} 
\centering 
\begin{tabularx} {\linewidth} { @ { } c @ { } c  @ { } @ { } @ { } r *{8} {R} }
\noalign {\vskip 1mm} 
\hline\hline 
\noalign {\vskip 1mm} 
Dimension & Degree & \multicolumn {9} {l} {Distance, $l$, from the reference vertex at level 0} \\ 
$f$ & $d$ & 2 & 3 & 4 & 5 & 6 & 7 & 8 & 9 & 10 \\
\hline 
\noalign {\vskip 1mm} 
1 & 2 & 2 & 2 & 2 & 2 & 2 & 2 & 2 & 2 & 2 \\
 & 3 & 4 & 4 & 4 & 4 & 4 & 4 & 4 & 4 & 4 \\
2 & 4 & 8 & 12 & 16 & 20 & 24 & 28 & 32 & 36 & 40 \\
& 5 & 12 & 20 & 28 & 36 & 44 & 52 & 60 & 68 & 76 \\
3 & 6 & 18 & 38 & 66 & 102 & 146 & 198 & 258 & 326 & 402 \\
& 7 & 24 & 56 & 104 & 168 & 248 & 344 & 456 & 584 & 728 \\
4 & 8 & 32 & 88 & 192 & 360 & 608 & 952 & 1408 & 1992 & 2720 \\
& 9 & 40 & 120 & 280 & 552 & 968 & 1560 & 2360 & 3400 & 4712 \\
\noalign {\vskip 2mm}
\hline
\end{tabularx}
\label{table:4D} 
\end{table}

Now these values are precisely the maximum size of each corresponding level of the distance partitions. So over the range of degrees and diameters considered, we can see that for each extremal and largest known circulant graph each distance partition level $l$ is filled to the maximum determined by the upper bound $M_{AC}(d,l)$. This suggests that not only the extremal graphs but also the graphs of degree 6 to 9 and the upper bound $M_{AC}(d,l)$ are optimal in some sense. Formulae for $LM_{AC}(d,l)$ as a function of $l$ are presented for degree 2 to 9 in Table \ref{table:3E}. For the graphs of dimension 1 and 2 all the levels are maximal, with the exception of the last level for degree 5. For dimension 3 we see that the proportion of maximal levels is about $2/3$ and for dimension 4 about $1/2$. For $f \ge 2$ we note that these proportions are represented by the expression $2/f$. The exact number of maximal levels in each case is also shown in Table \ref{table:3E}. These values are proved for degree 2, 4, 6 and 8 in the following six theorems. The proofs for odd degrees are similar.

\begin{table} [!htbp]
\small
\caption{\small Maximal distance partition levels of extremal and largest known circulant graphs of degree 2 to 9 and diameter $k$: size of each maximal level, $LM_{AC}(d,l)$, and the position of the last maximal level.} 
\centering 
\begin{tabular}{ @ { } c c c l l } 
\noalign {\vskip 2mm}  
\hline\hline 
\noalign {\vskip 1mm} 
Dimension, $f$ & Degree, $d$ & $LM_{AC}(d,1)$ & $LM_{AC}(d,l), l\geq 2$ & Last maximal level \\
\hline
\noalign {\vskip 1mm} 
1 & 2 & 2 & 2 & $k$ \\
& 3 & 3 & 4 & $k$\\
\noalign {\vskip 1mm} 
2 & 4 & 4 &  $4l$ & $k$ \\
& 5 & 5 & $8l-4$ & $k-1$ \\
\noalign {\vskip 1mm} 
3 & 6 & 6 & $4l^2+2$ & $\lfloor (2k+1)/3\rfloor$ \\
& 7 & 7 & $8l^2-8l+8$ & $\lfloor 2k/3\rfloor$ \\
\noalign {\vskip 1mm} 
4 & 8 & 8 & $(8l^3+16l)/3$ & $\lfloor (k+1)/2\rfloor$ \\
& 9 & 9 & $(16l^3-24l^2+56l-24)/3$ & $\lfloor k/2\rfloor$ \\
\hline
\end{tabular}
\label{table:3E} 
\end{table}

\begin{theorem}
For extremal degree $2$ circulant graphs of arbitrary diameter $k$, every distance partition level $l$ is maximal in the sense determined by the upper bound $M_{AC}(2,l)$.
\label{theorem:B0}
\end{theorem}
\begin{proof}
For an extremal circulant graph of degree 2 and arbitrary diameter $k$, the order is $CC(2,k)=2k+1$ and a generator set is $\{g_1\}$ where $g_1=1$. The graph is a cycle graph, and it is immediately evident that each level contains exactly 2 vertices and is therefore maximal.
\end{proof}

\begin{theorem}
For extremal degree $4$ circulant graphs of arbitrary diameter $k$, every distance partition level $l$ is maximal in the sense determined by the upper bound $M_{AC}(4,l)$.
\label{theorem:B1}
\end{theorem}
\begin{proof}
For an extremal circulant graph of degree 4 and arbitrary diameter $k$, the order is $CC(4,k)=n=2k^2+2k+1$ and a generator set is $\{g_1,g_2\}$ where $g_1=1$ and $g_2=2k+1$. Let $v_a$ and $v_b$ be vertices in level $l \leq k$, relative to an arbitrary root vertex. As cyclic groups are Abelian the generating elements for each vertex can be taken in any order. Therefore $v_a \equiv a_1g_1+a_2g_2$ (mod $n$) and $v_b \equiv b_1g_1+b_2g_2$ (mod $n$) where $\vert a_1\vert+\vert a_2\vert=\vert b_1\vert+\vert b_2\vert=l$ for some $a_1, a_2, b_1, b_2 \in \Z$. The level is maximal if and only if $a_1=b_1$ and $a_2=b_2$ whenever $v_a=v_b$.

Suppose that $v_a=v_b$. Then $(a_1-b_1)g_1+(a_2-b_2)g_2\equiv 0$ (mod $n$). Now $\vert (a_1-b_1)g_1+(a_2-b_2)g_2\vert \leq 2kg_2=4k^2+2k<2n$. Thus $(a_1-b_1)g_1+(a_2-b_2)g_2=0$ or $\pm n$. If  $(a_1-b_1)g_1+(a_2-b_2)g_2=0$ then $a_1-b_1\equiv0$ (mod $2k+1$). But $\vert a_1\vert$, $\vert b_1\vert \leq k$. Hence $a_1=b_1$ and so $a_2=b_2$. Therefore in this case level $l$ is maximal. If  $(a_1-b_1)g_1+(a_2-b_2)g_2=\pm n$, then without loss of generality assume  $(a_1-b_1)g_1+(a_2-b_2)g_2=n$. Noting that $n\equiv k+1$ (mod $2k+1$) we have $a_1-b_1\equiv k+1$ (mod $2k+1$), so that $a_1-b_1=-k$ or $k+1$. If $a_1-b_1=-k$ then $(a_2-b_2)g_2=n+k=2k^2+3k+1$ so that $a_2-b_2=k+1$. If $a_1-b_1=k+1$ then $(a_2-b_2)g_2=n-k-1=2k^2+k$ so that $a_2-b_2=k$. In either case $\vert a_1-b_1\vert+\vert a_2-b_2\vert=2k+1$ contrary to the premise. Hence we must have $a_1=b_1$ and $a_2=b_2$, and the level is maximal.
\end{proof}

\begin{theorem}
For largest known degree $6$ circulant graphs of arbitrary diameter $k$, every distance partition level $l$ is maximal for $l\leq\lfloor (2k+1)/3\rfloor$ in the sense determined by the upper bound $M_{AC}(6,l)$.
\label{theorem:B2}
\end{theorem}
\begin{proof}
We give the proof just for the case $k\equiv 0$ (mod 3), $k\geq 3$, for a graph of isomorphism class 1. The proofs for isomorphism class 2 and for cases $k\equiv 1$ and $k\equiv 2$ (mod 3) are similar. So let $k=3m$ for $m\geq 1$.
For a best circulant graph of degree 6 and arbitrary diameter $3m$, the order is $DF(6,k)=n=32m^3+16m^2+6m+1$ and a generator set for isomorphism class 1 is $\{g_1,g_2,g_3\}$ where $g_1=1$, $g_2=4m+1$ and $g_3=16m^2+4m+1$. Let $v_a$ and $v_b$ be vertices in level $l \leq 2m$, relative to an arbitrary root vertex.
Then $v_a \equiv a_1g_1+a_2g_2+a_3g_3$ (mod $n$) and $v_b \equiv b_1g_1+b_2g_2+b_3g_3$ (mod $n$) where $\vert a_1\vert+\vert a_2\vert+\vert a_3\vert=\vert b_1\vert+\vert b_2\vert+\vert a_3\vert=l\leq 2m$ for some $a_1, a_2, a_3, b_1, b_2, b_3 \in \Z$. Also let $w_a = a_1g_1+a_2g_2+a_3g_3$ and $w_b = b_1g_1+b_2g_2+b_3g_3$. The level is maximal if and only if $a_1=b_1$, $a_2=b_2$ and $a_3=b_3$ whenever $v_a=v_b$.

Suppose that $v_a=v_b$. Then $(a_1-b_1)g_1+(a_2-b_2)g_2+(a_3-b_3)g_3\equiv 0$ (mod $n$).
Now $lg_3\leq2mg_3=32m^3+8m^2+2m<n$. So $\vert w_a-w_b\vert<2n$, and $w_a-w_b=0$ or $\pm n$.

First we consider the case $w_a-w_b=\pm n$. Without loss of generality we may assume $w_a-w_b=n$, so that $(a_1-b_1)g_1+(a_2-b_2)g_2+(a_3-b_3)g_2=n$. Now $n=32m^3+16m^2+6m+1=2mg_3+2mg_2+(2m+1)g_1$. This is equivalent to $6m+1$ edges and there is no construction with fewer edges. However $w_a$ and $w_b$ are each the sum of at most $2m$ edges, or $4m$ in total. Hence there exists no pair $w_a$ and $w_b$ such that $w_a-w_b=n$. Therefore this case admits no solution.

Now we consider the case $w_a-w_b=0$. Here $(a_1-b_1)g_1+(a_2-b_2)g_2+(a_3-b_3)g_2=0$. Suppose for a contradiction that $a_3\neq b_3$. Then $\vert a_3-b_3\vert\geq1$, so that $\vert (a_1-b_1)g_1+(a_2-b_2)g_2\vert\geq g_3=16m^2+4m+1$. But $\vert (a_1-b_1)+(a_2-b_2)\vert\leq\vert a_1-b_1\vert+\vert a_2-b_2\vert\leq4m$, so that $\vert (a_1-b_1)g_1+(a_2-b_2)g_2\vert\leq 4mg_2=16m^2+4m$. This contradiction proves that $a_3=b_3$. Now suppose that $a_2\neq b_2$. Then $\vert a_2-b_2\vert\geq1$, so that $\vert a_1-b_1\vert g_1\geq g_2=4m+1$. But  $\vert a_1-b_1\vert g_1\leq 4m$. Hence $a_2=b_2$ and thus also $a_1=b_1$.
This proves that if $v_a=v_b$ then we must have $a_1=b_1$, $a_2=b_2$ and $a_3=b_3$. Hence the level is maximal.
\end{proof}

\begin{theorem}
For largest known degree $6$ circulant graphs of arbitrary diameter $k$, every distance partition level $l$ is submaximal for $l>\lfloor (2k+1)/3\rfloor$ in the sense determined by the upper bound $M_{AC}(6,l)$.
\label{theorem:C1}
\end{theorem}
\begin{proof}
We give the proof just for the case $k\equiv 0$ (mod 3), $k\geq 3$, for a graph of isomorphism class 1. The proofs for isomorphism class 2 and for cases $k\equiv 1$ and $k\equiv 2$ (mod 3) are similar. So let $k=3m$ for $m\geq 1$.
For a best circulant graph of degree 6 and arbitrary diameter $3m$, the order is $DF(6,k)=n=32m^3+16m^2+6m+1$ and a generator set for isomorphism class 1 is $\{g_1,g_2,g_3\}$ where $g_1=1$, $g_2=4m+1$ and $g_3=16m^2+4m+1$. Let $v_a$ and $v_b$ be vertices in level $2m+1$, relative to an arbitrary root vertex.
Then $v_a \equiv a_1g_1+a_2g_2+a_3g_3$ (mod $n$) and $v_b \equiv b_1g_1+b_2g_2+b_3g_3$ (mod $n$) where $\vert a_1\vert+\vert a_2\vert+\vert a_3\vert=\vert b_1\vert+\vert b_2\vert+\vert a_3\vert=2m+1$ for some $a_1, a_2, a_3, b_1, b_2, b_3 \in \Z$. The level is maximal if and only if $a_1=b_1$, $a_2=b_2$ and $a_3=b_3$ whenever $v_a=v_b$. Now consider the case where $a_1=0, a_2=2m+1, a_3=0, b_1=-1, b_2=-(2m-1), g_3=1$. Then $v_a=(2m+1)(4m+1)=8m^2+6m+1$ and $v_b=-1-(2m-1)(4m+1)+(16m^2+4m+1)=8m^2+6m+1$. Thus $v_a=v_b$ and so level $2m+1$ is submaximal. If a level is submaximal then the next level must necessarily also be submaximal. Therefore every level $l\geq 2m+1$ is submaximal.
\end{proof}

\begin{theorem}
For largest known degree $8$ circulant graphs of arbitrary diameter $k$, every distance partition level $l$ is maximal for $l\leq\lfloor (k+1)/2\rfloor$ in the sense determined by the upper bound $M_{AC}(8,l)$.
\label{theorem:B3}
\end{theorem}
\begin{proof}
We give the proof just for the case $k\equiv 0$ (mod 2), $k\geq 4$. The proof for the case $k\equiv 1$ (mod 2) is similar. So let $k=2m$ for $m\geq 2$.
For a largestknown circulant graph of degree 8 and arbitrary diameter $2m$, the order is $L(8,k)=n=8m^4+8m^3+12m^2+4m$ and a generator set is $\{g_1,g_2,g_3,g_4\}$ where $g_1=1$, $g_2=4m^3+4m^2+6m+1$, $g_3=4m^4+4m^2-4m$ and $g_4=4m^4+4m^2-2m$. Let $v_a$ and $v_b$ be vertices in level $l \leq m$, relative to an arbitrary root vertex. Then $v_a \equiv\sum_{i=1}^4{a_ig_i}$ (mod $n$) and $v_b \equiv\sum_{i=1}^4{b_ig_i}$ (mod $n$) where $\sum_{i=1}^4{\vert a_i\vert}=\sum_{i=1}^4{\vert b_i\vert}=l\leq m$ for some $a_i, b_i \in \Z$. The level is maximal if and only if $a_i=b_i$ for $i=1,2,3,4$ whenever $v_a=v_b$.

Suppose that $v_a=v_b$. Then $\sum_{i=1}^4{(a_i-b_i)g_i}=cn$ for some $c\in\Z$. We note that $n\equiv g_3\equiv g_4\equiv 0$ (mod $2m$) and $g_1\equiv g_2\equiv 1$ (mod $2m$). Hence $(a_1-b_1)+(a_2-b_2)\equiv 0$ (mod $2m$). First suppose $\vert (a_1-b_1)+(a_2-b_2)\vert =2m$. Then $l=m, \vert a_1+a_2\vert =m,  \vert b_1+b_2\vert =m$, and $a_3=a_4=b_3=b_4=0$. Without loss of generality suppose that $a_1+a_2=-m$ and $b_1+b_2=m$. Then $v_a-v_b\equiv (a_1-b_1)+((m-a_1)-(m-b_1))(4m^3+4m^2+6m+1)\equiv (b_1-a_1)(4m^3+4m^2+6m)$ (mod $n$). But $(b_1-a_1)(4m^3+4m^2+6m)\leq2m(4m^3+4m^2+6m)<n$. Thus $v_a-v_b=(b_1-a_1)(4m^3+4m^2+6m)=0$. Hence $a_1=b_1$. So $a_2=b_2$ and the level is maximal.

Now suppose that $(a_1-b_1)+(a_2-b_2)=0$, so that $(a_1-b_1)(g_1-g_2)+(a_3-b_3)g_3+(a_4-b_4)g_4=cn$. Note that $\sum_{i=1}^4{\vert a_i-b_i\vert}\leq2m$, and let $a_i-b_i=x_im+y_i$ for $x_i, y_i\in\Z$ for $i=1,3,4$. Then we have $-(x_1m+y_1)(4m^3+4m^2+6m)+(x_3m+y_3)(4m^4+4m^2-4m)+(x_4m+y_4)(4m^4+4m^2-2m)=c(8m^4+8m^3+12m^2+4m)$. Equating coefficients of powers of $m$;
$m^5$: $4x_3+4x_4=0$.
$m^4$: $-4x_1+4y_3+4y_4=8c$.
$m^3$: $-4y_1-4x_1+4x_3+4x_4=8c$.
$m^2$: $-4y_1-6x_1+4y_3-4x_3+4y_4-2x_4=12c$.
$m$:  $-6y_1-4y_3-2y_4=4c$.

As $x_4=-x_3$ and $\sum_{i=1}^4{\vert a_i-b_i\vert}\leq2m$ we must have $x_3=0$ or $x_3=1$.
If $x_3=0$ then $x_4=0, x_1=-2c, y_1=0, y_3=-2c$ and $y_4=2c$. Thus $\sum_{i=1}^4{\vert a_i-b_i\vert}=\vert 2cm\vert+\vert 2cm\vert+\vert 2c\vert+\vert 2c\vert=4\vert c\vert(m+1)\leq2m$. So $c=0$ and $a_i=b_i$ for $i=1,2,3,4$.
If $x_3=1$ then $x_4=-1, x_1=1-2c, y_1=-1, y_3=2-2c$ and $y_4=2c-1$. Thus $\sum_{i=1}^4{\vert a_i-b_i\vert}=\vert (1-2c)m-1\vert+\vert(1-2c)m-1\vert+\vert m+2-2c\vert+\vert -m+2c-1\vert$. If $c=0$, then $\sum_{i=1}^4{\vert a_i-b_i\vert}=4m+1$. If $c=1$, then $\sum_{i=1}^4{\vert a_i-b_i\vert}=4m$. For any other value of $c$, $\sum_{i=1}^4{\vert a_i-b_i\vert }\geq 8m$. However we have that $\sum_{i=1}^4{\vert a_i-b_i\vert }\leq 2m$. Hence there is no solution with $x_3=1$.
This completes the proof that level $l$ is maximal for any $l\leq m$.
\end{proof}

\begin{theorem}
For largest known degree $8$ circulant graphs of arbitrary diameter $k$, every distance partition level $l$ is submaximal for $l>\lfloor (k+1)/2\rfloor$ in the sense determined by the upper bound $M_{AC}(8,l)$.
\label{theorem:C2}
\end{theorem}
\begin{proof}

We give the proof just for the case $k\equiv 0$ (mod 2), $k\geq 4$. The proof for the case $k\equiv 1$ (mod 2) is similar. So let $k=2m$ for $m\geq 2$.
For a best circulant graph of degree 8 and arbitrary diameter $2m$, the order is $L(8,k)=n=8m^4+8m^3+12m^2+4m$ and a generator set is $\{g_1,g_2,g_3,g_4\}$ where $g_1=1$, $g_2=4m^3+4m^2+6m+1$, $g_3=4m^4+4m^2-4m$ and $g_4=4m^4+4m^2-2m$. Let $v_a$ and $v_b$ be vertices in level $m+1$, relative to an arbitrary root vertex. Then $v_a \equiv\sum_{i=1}^4{a_ig_i}$ (mod $n$) and $v_b \equiv\sum_{i=1}^4{b_ig_i}$ (mod $n$) where $\sum_{i=1}^4{\vert a_i\vert}=\sum_{i=1}^4{\vert b_i\vert}=m+1$ for some $a_i, b_i \in \Z$. The level is maximal if and only if $a_i=b_i$ for $i=1,2,3,4$ whenever $v_a=v_b$.
Now consider the case where $a_1=m, a_2=0, a_3=1, a_4=0, b_1=-m, b_2=0, g_3=0, b_4=1$. Then $v_a=m+(4m^3+4m^2-4m)=4m^3+4m^2-3m$ and $v_b=-m+(4m^3+4m^2-2m)=4m^3+4m^2-3m$. Thus $v_a=v_b$ and so level $m+1$ is submaximal. If a level is submaximal then the next level must necessarily also be submaximal. Therefore every level $l\geq m+1$ is submaximal.
\end{proof}


\section {Submaximal levels}

Now we consider the evolution of the size of the submaximal levels in a largest known graph's distance partition. Clearly the difference between the order of the graph, $DF(d,k)$ or $L(d,k)$, and the corresponding upper bound $M_{AC}(d,k)$ will be equal to the sum for all levels of the shortfall in each level relative to its maximum value.

So for a graph of degree $d$ and diameter $k$ we define the \it level defect \rm $D(d,k,l)$ for each level $l$ in its distance partition profile to be the difference between its actual size $s_l$ and the maximal size $LM_{AC}(d,l)$: $D(d,k,l) = LM_{AC}(d,l)-s_l$. Then for largest known graphs:
\[ \sum ^k _{l=0} D(d,k,l) =
\begin {cases}
 M_{AC}(d,k) - DF(d,k)& \mbox {for } d=6, 7 \\
 M_{AC}(d,k) - L(d,k) & \mbox {for } d = 8, 9. \\
\end {cases}
\]
For example for degree $d=6$
\[ \sum ^k _{l=0} D(6,k,l) = 
\begin {cases}
(4k^3+6k^2+18k)/27 & \mbox {for } k \equiv 0 \pmod 3 \\
(4k^3+6k^2-6k-4)/27 & \mbox {for } k \equiv 1 \pmod 3 \\
(4k^3+6k^2+18k+16)/27 & \mbox {for } k \equiv 2 \pmod 3. \\
\end {cases}
\]

Considering degree 6, for diameter $k \leq 15$, the submaximal level defects are shown in Table \ref {table:4Z}.

\begin{table}[!htbp]
\small
  \centering
  \caption{\small Submaximal level defects, $D(6,k,l)$, of the distance partition profiles for largest known graphs of degree 6}
\begin{tabularx} {\linewidth} { @ { } c  @ {} r *{16} {R} }
\noalign {\vskip 1mm}  
\hline\hline 
\noalign {\vskip 1mm} 
Diameter & Order & \multicolumn {11} {l} {Distance partition level, $l$} \\ 
$k$ & $n$ & 0 & 1 & 2 & 3 & 4 & 5 & 6 & 7 & 8 & 9 & 10 & 11 & 12 & 13 & 14 & 15 \\
\hline 
\noalign {\vskip 1mm} 
3 & 55    & 0 & 0     & 0     & 8     &       &       &       &       &       &       &       &       &       &       &       &  \\
4 & 117   & 0 & 0     & 0     & 0     & 12    &       &       &       &       &       &       &       &       &       &       &  \\
5 & 203   & 0 & 0     & 0     & 0     & 4     & 24    &       &       &       &       &       &       &       &       &       &  \\
6 & 333   & 0 & 0     & 0     & 0     & 0     & 8     & 36    &       &       &       &       &       &       &       &       &  \\
7 & 515   & 0 & 0     & 0     & 0     & 0     & 0     & 12    & 48    &       &       &       &       &       &       &       &  \\
8 & 737   & 0 & 0     & 0     & 0     & 0     & 0     & 4     & 24    & 68    &       &       &       &       &       &       &  \\
9 & 1027  & 0 & 0     & 0     & 0     & 0     & 0     & 0     & 8     & 36    & 88    &       &       &       &       &       &  \\
10 & 1393  & 0 & 0     & 0     & 0     & 0     & 0     & 0     & 0     & 12    & 48    & 108   &       &       &       &       &  \\
11 & 1815  & 0 & 0     & 0     & 0     & 0     & 0     & 0     & 0     & 4     & 24    & 68    & 136   &       &       &       &  \\
12 & 2329  & 0 & 0     & 0     & 0     & 0     & 0     & 0     & 0     & 0     & 8     & 36    & 88    & 164   &       &       &  \\
13 & 2943  & 0 & 0     & 0     & 0     & 0     & 0     & 0     & 0     & 0     & 0     & 12    & 48    & 108   & 192   &       &  \\
14 & 3629  & 0 & 0     & 0     & 0     & 0     & 0     & 0     & 0     & 0     & 0     & 4     & 24    & 68    & 136   & 228   &  \\
15 & 4431  & 0 & 0     & 0     & 0     & 0     & 0     & 0     & 0     & 0     & 0     & 0     & 8     & 36    & 88    & 164   & 264 \\
    \hline
    \end{tabularx}%
  \label{table:4Z}%
\end{table}%

Reading the columns of the table upward from the zero elements (where the level is maximal), there is one sequence of submaximal level defects for even levels: $4,12,36,68,108,164,$ $228,\ldots,$ and one for odd levels: $8,24,48,88,136,192,264,\ldots.$ The reason there are two different alternating sequences depends on the fact that for every increase of 3 in the diameter an extra two levels become maximal. The odd levels are related to an increase of 2 in the diameter and the even levels to an increase of 1. The second order differences of both theses sequences display the same repeating cycle of length three: $8,8,16,\ldots.$
Taken in reverse order these sequences define the evolution of a new level as the diameter is increased. For each level $l$ the level defect decreases in a fixed pattern depending on the parity of the level until the level reaches its maximal value $LM_{AC}(6,l)$ as defined by the upper bound $M_{AC}(6,l)$. Reading along each row of Table \ref {table:4Z}, these level evolutions translate into the submaximal level defects of the distance partition profile of a graph of the corresponding diameter. For the same reason that every increase of 3 in the diameter is paired with an extra two levels becoming maximal, we find three different submaximal level defect sequences, one for each value of the diameter $k \pmod 3 $:
\[  
\begin {cases}
8, 36, 88, 164, 264, \ldots & \mbox {for } k \equiv 0 \pmod 3 \\
12, 48, 108, 192, 300, \ldots & \mbox {for } k \equiv 1 \pmod 3 \\
4, 24, 68, 136, 228, \ldots & \mbox {for } k \equiv 2 \pmod 3. \\
\end {cases}
\]

The second order differences of all three sequences have the same constant value of 24. The level defects for degree 7 follow a similar form in having two alternating sequences along levels, both with second order differences displaying a repeating cycle of length three, $16, 16, 32,\ldots;$ and three along the profiles for each diameter, all with a constant second order difference of 48.

For degree 8 there is only one sequence of submaximal level defects valid for all levels: $8, 24, 64,$ $128, 232, 376, 576, 832, 1160, 1560,\ldots,$ with the exception that the final value in any of the sequences is always augmented by 1. The reason they are all the same, as distinct from the two versions for dimension 3, stems from the fact that for every increase of 2 in the diameter a single extra level becomes maximal. Therefore the context is similar for each level. The third order differences of this sequence form a repeated cycle of length two, $0,16,\dots,$ except that the final level in any sequence is 1 higher. For the same reason, reading along the profile for each diameter we find two different submaximal level defect sequences (instead of three as for dimension 3):
\[  
\begin {cases}
8, 64, 232, 576, 1160, 2048 \ldots & \mbox {for } k \equiv 0 \pmod 2 \\
24, 128, 376, 832, 1560, 2624 \ldots & \mbox {for } k \equiv 1 \pmod 2. \\
\end {cases}
\]
both with a constant third order difference of 64. Again the value of the final level is increased by 1. The level defects for degree 9 follow a similar form in having a single sequence along levels with third order difference displaying a repeating cycle of length two, $0, 32,\ldots;$ and two along the profiles for each diameter, both with a constant third order difference of 128.

The sequential second or third order differences between the size of maximal levels and submaximal level defects for degree 6 to 9 is summarised in Table \ref {table:4A}, along with a parametrisation in terms of the dimension $f$. For each degree the relevant order difference for submaximal level defects is greater than the size of the maximal levels by a multiple of the dimension, indicating a relation between the sequence of submaximal level defects and the Abelian Cayley upper bound function $M_{AC}(d,k)$.

\begin{table} [!htbp]
\small
\caption{\small Sequential differences of size of maximal levels and submaximal level defects for largest known graphs of degree 6 to 9, also parametrised in terms of the dimension $f$.} 
\centering 
\begin{tabular}{ @ { } c c c c c } 
\noalign {\vskip 2mm}  
\hline\hline 
\noalign {\vskip 1mm} 
Dimension, & Degree, & Order of & Difference for & Difference for \\
$f$ & $d$ & difference & maximal levels & submaximal level defects \\
\hline
\noalign {\vskip 1mm} 
3 & 6 & 2 & 8 & 24 \\
& 7 & 2 & 16 & 48 \\
\noalign {\vskip 1mm} 
4 & 8 & 3 & 16 & 64 \\
& 9 & 3 & 32 & 128 \\
\hline
\noalign {\vskip 2mm} 
$f$ & $2f$ & $f-1$ & $2^f$ & $f2^f$ \\
\noalign {\vskip 1mm}
& $2f+1$ & $f-1$ & $2^{f+1}$ & $f2^{f+1}$ \\
\noalign {\vskip 1mm} 
\hline
\end{tabular}
\label{table:4A} 
\end{table}


\section {Odd girth of extremal and largest known circulant graphs of degree 2 to 9}

Any circulant graph of degree $d\geq 3$ has at least two distinct generators which, taken as pairs in either order, generate two distinct paths of length 2 between a single pair of vertices. Hence these graphs have girth of at most 4. However it is noteworthy that the extremal and largest known circulant graphs of degree 2 to 9 all have odd girth that is maximal for their diameter. The proofs that such graphs of diameter $k$ have odd girth $2k+1$ are presented for even degree only, but the proofs for odd degree are similar. On the other hand, the circulant graphs corresponding to Chen and Jia's lower bound have lower odd girth: for example, odd girth $k$ for degree 8 where $k\equiv 1 \pmod 4$, and $(4k+1)/5$ for degree 10 where $k\equiv 1\pmod 5$.

\begin{theorem}
For extremal degree $2$ circulant graphs of arbitrary diameter $k$ their odd girth is maximal, $2k+1$.
\label{theorem:D0}
\end{theorem}
\begin{proof}
For an extremal circulant graph of degree 2 and arbitrary diameter $k$, the order is $CC(2,k)=n=2k+1$ and a generator set is $\{g_1\}$ where $g_1=1$. The graph is a cycle graph, and it is immediately clear that it has girth $2k+1$ and therefore also odd girth $2k+1$.
\end{proof}

\begin{theorem}
For extremal degree $4$ circulant graphs of arbitrary diameter $k$ their odd girth is maximal, $2k+1$.
\label{theorem:D1}
\end{theorem}
\begin{proof}
For an extremal circulant graph of degree 4 and arbitrary diameter $k$, the order is $CC(4,k)=n=2k^2+2k+1$ and a generator set is $\{g_1,g_2\}$ where $g_1=1$ and $g_2=2k+1$.
If the graph contains a cycle of length $2l+1$ then $d_1g_1+d_2g_2\equiv0$ (mod $n$) with $\vert d_1\vert+\vert d_2\vert\leq 2l+1$ and $\vert d_1\vert+\vert d_2\vert\equiv1$ (mod 2) for some $d_1, d_2\in\Z$.

When $l=k$ we have two solutions: $d_1=-k, d_2=k+1$, and $d_1=k+1, d_2=k$, both giving $d_1g_1+d_2g_2=n$. Therefore the graph contains cycles of length $2k+1$. Now suppose $l<k$, so that $\vert d_1\vert+\vert d_2\vert\leq 2k-1$. Then $\vert d_1g_1+d_2g_2\vert\leq(2k-1)g_2=(2k-1)(2k+1)=4k^2-1<2n$, so that $d_1g_1+d_2g_2=0$ or $\pm n$. First consider the case $d_1g_1+d_2g_2=0$. We note that $g_1\equiv1$ (mod $2k+1$) and $g_2\equiv0$ (mod $2k+1$). So $d_1\equiv0$ (mod $2k+1$), and hence $d_1=0$. Therefore $d_2=0$ and there is no odd cycle of length $2l+1$. Now consider $d_1g_1+d_2g_2=\pm n$. Without loss of generality we assume  $d_1g_1+d_2g_2=n$. Noting that $n\equiv k+1$ (mod $2k+1$) we have $d_1\equiv k+1$ (mod $2k+1$). Hence $d_1=-k$ or $d_1=k+1$. If $d_1=-k$ then $d_2g_2=2k^2+3k+1$ and $d_2=k+1$. If $d_1=k+1$ then $d_2g_2=2k^2+k$ and $d_2=k$. In both cases $\vert d_1\vert+\vert d_2\vert=2k+1$. Therefore there are no odd cycles of length $2l+1$ for $l<k$.
\end{proof}

\begin{theorem}
For largest known degree $6$ circulant graphs of arbitrary diameter $k$ their odd girth is maximal, $2k+1$.
\label{theorem:D2}
\end{theorem}
\begin{proof}
We give the proof just for the case $k\equiv 0$ (mod 3), $k\geq 3$ for isomorphism class 1. The proofs for isomorphism class 2 and for cases $k\equiv 1$ and $k\equiv 2$ (mod 3) are similar. So let $k=3m$ for $m\geq 1$.
For a largest known circulant graph of degree 6 and arbitrary diameter $3m$, the order is $DF(6,k)=n=32m^3+16m^2+6m+1$ and a generator set is $\{g_1,g_2,g_3\}$ where $g_1=1$, $g_2=4m+1$ and $g_3=16m^2+4m+1$.

If the graph contains a cycle of length $2l+1$ then $d_1g_1+d_2g_2+d_3g_3\equiv0$ (mod $n$) with $\vert d_1\vert+\vert d_2\vert+\vert d_3\vert\leq 2l+1$ and $\vert d_1\vert+\vert d_2\vert+\vert d_3\vert\equiv1$ (mod 2) for some $d_1, d_2, d_3\in\Z$. When $l=3m$ we have a solution: $d_1=2m+1, d_2=2m$ and $d_3=2m$, giving $d_1g_1+d_2g_2+d_3g_3=n$. Therefore the graph contains a cycle of length $6m+1$. Now suppose $l<3m$, so that $\vert d_1\vert+\vert d_2\vert+\vert d_3\vert\leq 6m-1$. Then $\vert d_1g_1+d_2g_2+d_3g_3\vert\leq(6m-1)g_3=(6m-1)(16m^2+4m+1)=96m^3+8m^2+2m-1<3n$. So $d_1g_1+d_2g_2+d_3g_3=0, \pm n$ or $\pm 2n$. But if $d_1g_1+d_2g_2+d_3g_3=0$ or $\pm 2n$ then $d_1g_1+d_2g_2+d_3g_3\equiv 0$ (mod 2), whereas $g_1\equiv g_2\equiv g_3\equiv 1$ (mod 2) and $\vert d_1\vert+\vert d_2\vert+\vert d_3\vert=1$ (mod 2), which gives a contradiction and thus admits no solution. So consider $d_1g_1+d_2g_2+d_3g_3=\pm n$. Then without loss of generality we may assume $d_1g_1+d_2g_2+d_3g_3=n$. Now $n=32m^3+16m^2+6m+1=2mg_3+2mg_2+(2m+1)g_1$. This is equivalent to $6m+1$ edges and there is no construction with fewer edges. Thus the graph has no odd cycle with length shorter than $6m+1$.
\end{proof}

\begin{theorem}
For largest known degree $8$ circulant graphs of arbitrary diameter $k$ their odd girth is maximal, $2k+1$.
\label{theorem:D3}
\end{theorem}
\begin{proof}
We give the proof just for the case $k\equiv 0$ (mod 2), $k\geq 4$. The proof for the case $k\equiv 1$ (mod 2) is similar. So let $k=2m$ for $m\geq 2$.
For a largest known circulant graph of degree 8 and arbitrary diameter $2m$, the order is $L(d,k)=n=8m^4+8m^3+12m^2+4m$ and a generator set is $\{g_1,g_2,g_3,g_4\}$ where $g_1=1$, $g_2=4m^3+4m^2+6m+1$, $g_3=4m^4+4m^2-4m$ and $g_4=4m^4+4m^2-2m$.

If the graph contains a cycle of length $2l+1$ then $\sum_{i=1}^4{d_ig_i}=cn$ with $\sum_{i=1}^4{\vert d_i\vert}\leq 2l+1$ and $\sum_{i=1}^4{\vert d_i\vert}\equiv1$ (mod 2) for some $d_i, c\in\Z$. Let $d_i=x_im+y_i$ for some $x_i, y_i\in\Z$ for $i=1, 2, 3, 4$, and let $c=x_5+y_5$ for some $x_5, y_5\in\Z$. So $\sum_{i=1}^4{\vert x_i\vert}\leq4$.
Then we have $(x_1m+y_1)+(x_2m+y_2)(4m^3+4m^2+6m+1)+(x_3m+y_3)(4m^4+4m^2-4m)+(x_4m+y_4)(4m^4+4m^2-2m)=(x_5m+y_5)(8m^4+8m^3+12m^2+4m)$. Equating coefficients of powers of $m$;
$m^5$: $4x_3+4x_4=8x_5$. So $x_3+x_4$ is even. Hence $x_3+x_4=-4, 2, 0, 2$ or $4$.
$m^4$: $4x_2+4y_3+4y_4=8x_5+8y_5$.
$m^3$: $4x_2+4y_2+4x_3+4x_4=12x_5+8y_5$.
$m^2$: $6x_2+4y_2-4x_3+4y_3-2x_4+4y_4=4x_5+12y_5$.
$m^1$:  $x_1+x_2+6y_2-4y_3-2y_4=4y_5$.
$m^0$:  $y_1+y_2=0$.
Noting that $n\equiv g_3 \equiv g_4 \equiv 0$ (mod $2m$) and $g_1 \equiv g_2 \equiv 1$ (mod $2m$), we have $d_1+d_2 \equiv 0$ (mod $2m$). Thus $x_1+x_2=-4, -2, 0, 2$ or $4$. By straighforward manipulation we find that there are no solutions for odd cycle length less than $4m+1$. For length $4m+1$ there are exactly 16 distinct combinations of generator elements corresponding to each combination of $x_i=\pm 1$ for $i=1,2,3,4$. For example with $x_1=x_2=x_3=x_4=1$ we have $y_1=y_2=y_3=0$ and $y_4=1$, giving $\sum_{i=1}^4{d_ig_i}=8m^5+8m^4+12m^3+4m^2=mn$. This proves the graph has odd girth $2k+1$.
\end{proof}


\section {Maximal distance partition levels by vertex type}

We have seen that all the extremal and largest known circulant graphs of degree 2 to 9 of arbitrary diameter $k$, above some threshold, have maximal odd girth, $2k+1$. This means that only in the final distance partition level, $k$, relative to an arbitrary root vertex, is any vertex adjacent to another in the same level. Thus any vertex in level $l$ for $1\leq l\leq k-1$ is adjacent only to vertices in level $l-1$ or $l+1$ and to none in level $l$. In this sense such a vertex may be defined as $thin$. As the degree, $d$, of each vertex is fixed, if it is adjacent to $s$ vertices in level $l-1$ then it must be adjacent to $d-s$ in level $l+1$, and such a vertex is defined to be of type $T_s$. Vertices in level $k$ may be adjacent to others in the same level but not of course to any in a further level. Therefore the type of these vertices is also well-defined by the number of adjacent vertices in the preceding level.


Analysis of the extremal and largest known circulant graphs of degree 2 to 9 reveals a regular structure in the number of vertices of each type in each distance partition level. Examples for graphs of degree 4, 6 and 8, all with diameter 12, are shown in Tables \ref{table:7A}, \ref{table:7B} and \ref{table:7C}.

\begin{table}[h]
\footnotesize
  \centering
\setlength{\tabcolsep}{5pt}
  \caption{\small Distance partition profile by vertex type: extremal graph of degree 4 and diameter 12}
\begin{tabularx}  {\linewidth} {@ {} c *{14} r l}
\noalign {\vskip 1mm}  
\hline\hline 
\noalign {\vskip 1mm}  
Vertex & \multicolumn {13} {l} {Distance partition level} & & Differences \\ 
type   & 0 & 1 & 2 & 3 & 4 & 5 & 6 & 7 & 8 & 9 & 10 & 11 & 12 & & Maximal \\
\hline 
\noalign {\vskip 1mm}  
$T_0$     & 1   &      &      &      &     &       &       &      &       &       &       &       &       &     &      \\
$T_1$     &     & 4    & 4   & 4   & 4   & 4    & 4    & 4   &  4    &  4   &  4   &  4   &   4   &  &  $\Delta^0=4$   \\
$T_2$     &     &      & 4   & 8   & 12  & 16  & 20  & 24 & 28   & 32  & 36   & 40  & 44  &   & $\Delta^1=4$    \\
\hline
\noalign {\vskip 1mm}  
Total    & 1  & 4  & 8  & 12  & 16  & 20  & 24 & 28 & 32  & 36  & 40 & 44  & 48   &  &  $\Delta^1=4$   \\
\hline
    \end{tabularx}%
  \label{table:7A}%
\end{table}%

\begin{table}[h]
\footnotesize
  \centering
\setlength{\tabcolsep}{4pt}
  \caption{\small Distance partition profile by vertex type: largest known graph of degree 6, diameter 12}
    \begin{tabularx} {\linewidth} { @ {} c  *{14} r l @ { } l }
\noalign {\vskip 1mm}  
\hline\hline 
\noalign {\vskip 1mm}  
Vertex & \multicolumn {13} {l} {Distance partition level} &  \multicolumn {3} {c} {Differences}  \\ 
type   & 0 & 1 & 2 & 3 & 4 & 5 & 6 & 7 & 8 & 9 & 10 & 11 & 12 &  & Maximal & Submaximal \\
\hline 
\noalign {\vskip 1mm}  
$T_0$     & 1   &      &      &      &      &       &       &       &       &       &       &       &       &            &            &                \\
$T_1$     &     & 6    & 6   & 6    & 6   & 6    & 6    & 6     &  6   &       &       &       &   &   &  $\Delta^0=6$  &  $\Delta^0=0$      \\
$T_2$     &     &      & 12  & 24  & 36  & 48  & 60  & 72   & 84  & 88   & 76   & 64  & 52  & &  $\Delta^1=12$ &  $\Delta^1=-12$   \\
$T_3$     &     &      &      & 8    & 24  & 48  & 80  & 120  &168 &226  &274  &306  &322 & &  $\Delta^2=8$   &  $\Delta^2=-16$   \\
$T_4$     &    &       &      &       &      &      &       &        &      & 4     & 16   & 28  & 40  &            &            &  $\Delta^1=12$     \\
\hline
\noalign {\vskip 1mm}  
Total    & 1  & 6  & 18  & 38  & 66  & 102  & 146 & 198 & 258  & 318  & 366  & 398  & 414   &  &  $\Delta^2=8$  &  $\Delta^2=-16$  \\
\hline
    \end{tabularx}%
  \label{table:7B}%
\end{table}%

\begin{table}[h]
\footnotesize
  \centering
\setlength{\tabcolsep}{4pt}
  \caption{\small Distance partition profile by vertex type: largest known graph of degree 8, diameter 12}
    \begin{tabularx} {\linewidth} {@ { } c *{14}  r l @ { } l}
\noalign {\vskip 1mm}  
\hline\hline 
\noalign {\vskip 1mm}  
Vertex & \multicolumn {13} {l} {Distance partition level} & \multicolumn {3} {c} {Differences} \\ 
type   & 0 & 1 & 2 & 3 & 4 & 5 & 6 & 7 & 8 & 9 & 10 & 11 & 12 & & Maximal & Submaximal \\
\hline 
\noalign {\vskip1mm}  
$T_0$     & 1   &      &      &      &     &       &       &       &       &       &       &       &       &  &      &              \\
$T_1$     &    & 8     & 8     & 8    & 8     & 8    & 8      & 4      &       &       &       &       &    &   &  $\Delta^0=8$     &  $\Delta^0=0$           \\$T_2$     &    &      & 24   & 48   & 72   & 96  & 120  & 136 & 124  & 100  & 76  & 52  & 26  &    & $\Delta^1=24$     &  $\Delta^1=-24$   \\
$T_3$    &    &     &      & 32     & 96     & 192     & 320   & 476   &  624  & 720   & 752   & 720   & 624  &   & $\Delta^2=32$     &  $\Delta^2=-64$    \\
$T_4$     &   &     &      &      & 16    &  64   & 160   & 328   & 564  & 844   & 1124  & 1356  & 1495  &  &  $\Delta^3=16$     &  $\Delta^3=-48$     \\
$T_5$     &   &     &      &     &      &     &      &   &  32     & 96      & 192      & 320      & 476 &    &  &  $\Delta^2=32$     \\
$T_6$    &   &    &      &      &      &      &     &     &   &    &   &       &   2    &       &       &       \\
\hline
\noalign {\vskip 1mm}  
Total    & 1  & 8  & 32  & 88  & 192  & 360  & 608 & 944 & 1344  & 1760  & 2144  & 2448  & 2623  & &  $\Delta^3=16$     &  $\Delta^3=-48$     \\
\hline
    \end{tabularx}%
  \label{table:7C}%
\end{table}%

In these examples for graphs of even degree we can see that the number of different vertex types increases with degree. Graphs of every degree contain type $T_1$ vertices, with their number remaining constant and equal to the degree within the maximal zone, while becoming absent from the submaximal zone within the first few levels. The number of each subsequent vertex type grows with a constant higher order difference in the maximal zone before reversing with a constant difference of the same order in the submaximal zone. In each case the values of these constants depend only on the degree and are independent of diameter. The graphs of odd degree display a similar structure. Common parameters of all the distance partition profiles by vertex type are presented in Table \ref{table:7D} for the maximal levels and in Table \ref{table:7E} for the submaximal levels.

\begin{table} [!htbp]
\small
\caption{\small Number of vertices of each type within maximal levels of extremal and largest known graphs of degree $d=4$ to $9$ (dimension $f=2$ to $4$).} 
\centering 
\begin{tabular}{ @ { } c c c c c } 
\noalign {\vskip 1mm}  
\hline\hline 
\noalign {\vskip 1mm}  
Vertex & \multicolumn {2} {c} {Even degree $d=2f$} & \multicolumn {2} {c} {Odd degree $d=2f+1$} \\
type $T_s$ & Difference order & Value & Difference order & Value \\
\hline
\noalign {\vskip 2mm}  
$1\leq s\leq f$ & $\Delta^{s-1}$ & $2^s {f\choose s}$ & $\Delta^{s-1}$ & $2^s {f\choose s}$ \\
\noalign {\vskip 2mm}  
$s=f+1$& - & - & $\Delta^{f-1}$ & $2^f $\\
\noalign {\vskip 1mm}  
\hline
\end{tabular}
\label{table:7D} 
\end{table}

\begin{table} [!htbp]
\small
\caption{\small Number of vertices of each type within submaximal levels of extremal and largest known graphs of degree $d=6$ to $9$ (dimension $f=3$ and $4$).} 
\centering 
\begin{tabular}{ @ { } c c c c c } 
\noalign {\vskip 1mm}  
\hline\hline 
\noalign {\vskip 1mm}  
Vertex & \multicolumn {2} {c} {Even degree $d=2f$} & \multicolumn {2} {c} {Odd degree $d=2f+1$} \\
type $T_s$ & Difference order & Value & Difference order & Value \\
\hline
\noalign {\vskip 2mm}  
$2\leq s\leq f$ & $\Delta^{s-1}$ & $-(s-1)2^s {f\choose s}$ & $\Delta^{s-1}$ & $-(s-1)2^s {f\choose s}$ \\
\noalign {\vskip 2mm}  
$s=f+1$&  $\Delta^{f-2}$ & $f2^{f-1} $ & $\Delta^{f-1}$ & $-(f-1)2^f $\\
\noalign {\vskip 1mm}
$s=f+2$& - & - & $\Delta^{f-2}$ & $f2^{f-1} $\\
\hline
\end{tabular}
\label{table:7E} 
\end{table}

We will now prove for circulant graphs of any degree $d$ and arbitrary diameter $k$ that the number of vertices of each type in each maximal level is determined by the same upper bound $M_{AC}(d,k)$ that has been proven to determine the total number of vertices in each maximal level.

\begin{theorem}
For circulant graphs of any degree $d$, the number of vertices, $VT(d,s,l)$, of type $T_s$ in distance partition level $l\geq1$, where the level is maximal, so that the size of the level is $LM_{AC}(d,l)=M_{AC}(d,l)-M_{AC}(d,l-1)$, is given by the following formulae.

For even degree $d=2f$ where $f$ is the dimension, we have:
\[VT(d,s,l) = {f \choose s}\sum_{i=1}^s{(-1)^{s-i}{s\choose i}LM_{AC}(2i,l)} \]
and hence
\[VT(d,s,l) =
\begin{cases}
\ d &\mbox{ for } s=1 \\
\noalign {\vskip 1mm}
\ 2fs\prod_{i=1}^{s-1}{2(f-i)(l-i)/(i+1)^2} &\mbox { for } s\geq 2.
\end{cases}
\]

For odd degree $d=2f+1$ where $f$ is the dimension, we have the recurrence relation:
\[VT(d,s,l)=VT(2f,s,l)+VT(2f,s-1,l-1) \mbox { for } s,l \geq 2 \]
and hence
\[VT(d,s,l) =
\begin{cases}
\ d &\mbox{ for } s=1 \\
\noalign {\vskip 1mm}
\ 2f(f-1)(l-1)+2f &\mbox{ for } s=2 \\
\noalign {\vskip 1mm}
\ 2fs\prod_{i=1}^{s-1}{2(f-i)(l-i)/(i+1)^2} \\
\ + 2f(s-1)\prod_{i=1}^{s-2}{2(f-i)(l-1-i)/(i+1)^2} &\mbox { for } s\geq 3.
\end{cases}
\]

\label{theorem:7F}
\end{theorem}
\begin{proof}
First consider a circulant graph of even degree $d$ and diameter $k$, being the Cayley graph of a cyclic group with generator set $G=\{ g_1, ..., g_f\}$ where $f=d/2$ is the dimension of the graph. Then the connection set is $C=\{\pm g_1, ..., \pm g_f\}$.

Let $v$ be a vertex at distance $l<k$ from an arbitrary root vertex $u$. Suppose for a contradiction that a path of length $l$ from $u$ to $v$ contains an edge generated by $g_i$ and another edge generated by $-g_i$ for some $g_i\in G$. Then as the group is Abelian the path from $u$ generated from the same set of edges after removing this pair would lead to $v$ after a distance of only $l-2$, contradicting the premise that $v$ is distant $l$ from $u$. Hence if $v$ is distant $l$ from $u$ then for any $g_i\in G$ no path of length $l$ from $u$ to $v$ contains edges generated by both $g_i$ and $-g_i$.

Suppose there exists a path $p$ of length $l\geq 2$ from $u$ to $v$ with two of its edges generated by different generators, say $c_1, c_2$ where $\vert c_1\vert = g_i$ and $\vert c_2\vert = g_j$ for some $i, j$ with $1 \leq i < j \leq f$. Then as the group is Abelian we may reorder the edges of $p$ to generate two distinct paths $p_1=(x_1, ..., x_{l-2}, c_1, c_2)$ and $p_2=(x_1, ..., x_{l-2}, c_2, c_1)$ from $u$ to $v$. Now consider the two vertices $v_1, v_2$ both distant $l-1$ from $u$, reached by following paths $p'_1=(x_1, ..., x_{l-2}, c_1)$ and $p'_2=(x_1, ..., x_{l-2}, c_2)$ from $u$. These are distinct vertices within distance partition level $l-1$ that are adjacent to $v$ in level $l$. Thus $v$ is connected to more than one level $l-1$ vertex and so is not a type $T_1$ vertex. Therefore for any type $T_1$ vertex in level $l\geq 1$ there is only one path from $u$ of length $l$ and each edge of the path is generated by the same element of the connection set.
Also every vertex on this path is also a type $T_1$ vertex generated by the same element. Conversely every element of the connection set generates a unique path from $u$ passing through vertices which are all distinct type $T_1$ vertices while the distance partition level remains in the maximal zone by the definition of the upper bound $M_{AC}(d,l)$. Therefore within the maximal zone the number of type $T_1$ vertices in each level, $VT(d,1,l)$, will be equal to the degree $d$ of the graph. This can also be expressed as the product of the number of such vertices for each generator, $LM_{AC}(2,l)=2$ where $LM_{AC}(d,l)=M_{AC}(d,l)-M_{AC}(d,l-1)$ as defined in an earlier section, and the number of generators, $f$, giving $VT(d,1,l)=2f=d$.

Now consider any two generators $g_i, g_j \in G$ and all vertices in level $l$ that can be reached from root vertex $u$ with a path of length $l$ comprised only of edges $\pm g_i$ and $\pm g_j$. As level $l$ is maximal, by definition of the upper bound $M_{AC}(d,k)$ there are $LM_{AC}(4,l)$ such vertices. We now restrict the vertex set to only those vertices where the path includes at least one edge $\pm g_i$ and one edge $\pm g_j$, so that $l\geq 2$. As the group is Abelian, each of these vertices will have at least one path from $u$ with final edge $\pm g_i$ and at least one path with final edge $\pm g_j$, and clearly no paths with any other final edge. Therefore all these vertices are of type $T_2$. The vertices with paths only of edges $\pm g_i$ or only of edges $\pm g_j$ are excluded. Thus the number of excluded vertices is $2LM_{AC}(2,l)$, and so the number of $T_2$ vertices in level $l$ reached by paths generated by the pair $g_i, g_j$ is given by $LM_{AC}(4,l)-2LM_{AC}(2,l)=4(l-1)$. As there are $f(f-1)/2$ distinct pairs of generators, the total number of $T_2$ vertices in level $l$ is given by $VT(d,2,l)=f(f-1)/2\times 4(l-1)=d(d-2)(l-1)/2$.

Similarly the number of vertices of type $T_3$ in level $l\geq 3$ from any given triad of generators is $LM_{AC}(6,l)-3LM_{AC}(4,l)+3LM_{AC}(2,l)=4(l-1)(l-2)$. As there are $f(f-1)(f-2)/6$ distinct triads of generators, the total number of type $T_3$ vertices in level is $VT(d,3,l)=d(d-2)(d-4)(l-1)(l-2)/12$. Also the number of vertices of type $T_4$ in level $l\geq 4$ from any given set of four generators is $LM_{AC}(8,l)-4LM_{AC}(6,l)+6LM_{AC}(4,l)-4LM_{AC}(2,l)=8(l-1)(l-2)(l-3)/3$, and so the total number of type $T_4$ vertices in level $l$ is $VT(d,4,l)=d(d-2)(d-4)(d-6)(l-1)(l-2)(l-3)/144$. More generally, for even degree $d=2f$ and any $s\geq 1$, \[VT(d,s,l) = {f \choose s}\sum_{i=1}^s{(-1)^{s-i}{s\choose i}LM_{AC}(2i,l)}.\]

This can be reformulated as: $VT(d,s,l)=ds\prod_{i=1}^{s-1}{(d-2i)(l-i)/(i+1)^2}$ for $s\geq 2$.

Now consider a circulant graph of odd degree $d=2f+1$ where $f$ is the dimension, and order $n$. If the generator set is $G=\{ g_1, ..., g_f\}$ then the connection set will be $C=\{\pm g_1, ..., \pm g_f, n/2\}$. As the element $n/2$ has order 2, it can only generate a path of length 1 to create one additional type $T_1$ vertex in level 1 but no extra vertices of type $T_1$ in any higher levels. Consider any level $l\geq 2$ within the maximal zone, and any vertex $v$ in this level. It is possible to reach $v$ by a path comprised either of edges generated by the non order 2 elements $\{\pm g_1,\ldots,\pm g_f\}$ alone or else also by including a single edge generated by the order 2 element $n/2$. It is not possible to reach any given vertex $v$ via paths of both cases as the level is within the maximal zone. So for any $s\geq 2$ the total number of vertices of type $T_s$ in level $l$ is the sum of the vertices reached by paths of either case. The number in the first case is simply the result just determined for a graph of even degree $d=2f$: $VT(2f,s,l)$. For the vertices in the second case, where the path includes an edge $n/2$, as the group is Abelian we may consider only those paths where the final edge is $n/2$. As vertex $v$ is of type $T_s$ in level $l$, then the preceding vertex $v'$ on each path must be of type $T_{s-1}$ in level $l-1$, where the path to $v'$ is comprised of edges from the connection set $C=\{\pm g_1, ..., \pm g_f\}$. Therefore, invoking the result for even degree again, the number of vertices in this case is $VT(2f,s-1,l-1)$. Hence for a circulant graph of odd degree $d=2f+1$ we have $VT(d,s,l)=VT(2f,s,l)+VT(2f,s-1,l-1)$ for $s,l\geq 2$.
\end{proof}

We now consider the number of type $T_1$ vertices in extremal, largest known and lower bound circulant graphs of degree 2 to 10. For these graphs the total number of type $T_1$ vertices is a linear function of the diameter in each case. For degree 2 there are two vertices in each distance partition level after the root and, being a cycle graph, they are all type $T_1$. For degree 3 there are four vertices in each level from level 2; two type $T_1$ and two type $T_2$. For degree $d\geq 4$ the total number of $T_1$ vertices increases by 4 for each increase of 1 in the diameter. This surprising result is true not only for the extremal and largest known graphs but also for the Chen and Jia lower bound constructions. The functions are summarised in Table \ref{table:7G}. The proofs for certain degree 6 and 8 cases are given in Theorems \ref{theorem:7G} and \ref{theorem:7H}. The other cases are proved similarly.

\begin{table} [h]
\small
\caption{\small Total number of type $T_1$ vertices in extremal, largest known and lower bound graphs up to degree 10 for arbitrary diameter $k\geq 2$.} 
\centering 
\begin{tabular}{ @ { } c l l c l } 
\noalign {\vskip 2mm}  
\hline\hline 
\noalign {\vskip 1mm}  
Degree, $d$ & Status & Order &  Isomorphism class & Number of $T_1$ vertices \\
\hline
\noalign {\vskip 1mm}  
2 & Extremal & $CC(2,k)$ & - & $2k$ \\
3 & Extremal & $CC(3,k)$ & - & $2k+1$ \\
\noalign {\vskip 2mm}  
4 & Extremal & $CC(4,k)$ & - & $4k$ \\
5 & Extremal & $CC(5,k)$ & - & $4k-1$ \\
\noalign {\vskip 2mm}  
6 & Largest known & $DF(6,k)$ & 1 & $4k$ \\
6 & Largest known & $DF(6,k)$ & 2 & $4k+2$ \\
\noalign {\vskip 2mm}  
7 & Largest known & $DF(7,k)$ & 1 & $4k-1$ \\
7 & Largest known & $DF(7,k)$ & 2 & $4k+1$ \\
\noalign {\vskip 2mm}  
8 & Largest known & $L(8,k)$ & - & $4k+4$ \\
\noalign {\vskip 2mm}  
9 & Largest known & $L(9,k)$ & 1 & $4k+1$ \\
9 & Largest known & $L(9,k)$ & 2 & $4k+1$ \\
\noalign {\vskip 3mm}  
8 & Lower bound & $CJ(8,k)$ & - & $4k-10$ \\
10 & Lower bound & $CJ(10,k)$ & - & $4k-10$ \\
\noalign {\vskip 1mm}  
\hline

\end{tabular}
\label{table:7G} 
\end{table}

\begin{theorem}
For largest known degree $6$ circulant graphs of diameter $k$, the total number of type $T_1$ vertices is equal to $4k$ for isomorphism class $1$ and equal to $4k+2$ for isomorphism class $2$.

\label{theorem:7G}
\end{theorem}
\begin{proof}
We give the proof just for the case $k\equiv 0$ (mod 3), $k\geq 3$, both for a graph of isomorphism class 1 and of isomorphism class 2. The proofs for cases $k\equiv 1$ and $k\equiv 2$ (mod 3) are similar. So let $k=3m$ for $m\geq 1$.
By Theorem \ref{theorem:B2} every distance partition level $l$ of a largest known degree 6 graph of diameter $k$ is maximal for $l\leq\lfloor (2k+1)/3\rfloor=2m$. Then by Theorem \ref{theorem:7F} the number of type $T_1$ vertices in level $l$ is equal to the degree $d=6$ for $1\leq l \leq 2m$.

For a largest known circulant graph of degree 6 of isomorphism class 1 and arbitrary diameter $3m$, the order is $DF(6,k)=n=32m^3+16m^2+6m+1$ and a generator set is $\{g_1,g_2,g_3\}$ where $g_1=1$, $g_2=4m+1$ and $g_3=16m^2+4m+1$.
Now we consider the six potential type $T_1$ vertices within level $2m+1$: $(2m+1)(\pm g_1)$, $(2m+1)(\pm g_2)$ and $(2m+1)(\pm g_3)$, relative to root vertex 0. First consider the vertex $g_2-2mg_1$. It can be reached by a path of length $2m+1$ with edges from two distinct generators. Therefore if it is in level $2m+1$ then it is a type $T_2$ vertex or possibly $T_3$ but not $T_1$. But we have $g_2-2mg_1=4m+1-2m=2m+1=(2m+1)g_1$. Thus $(2m+1)g_1$ is not a type $T_1$ vertex in level $2m+1$, and so neither is $(2m+1)(-g_1)$.
Similarly the vertex $g_3-(2m-1)g_2-g_1$ can be reached by a path of length $2m+1$ with edges from three distinct generators. Therefore if it is in level $2m+1$ then it is a type $T_3$ vertex. Now $g_3-(2m-1)g_2-g_1=16m^2+4m+1-(2m-1)(4m+1)-1=8m^2+6m+1=(2m+1)(4m+1)=(2m+1)g_2$. Thus $(2m+1)g_2$ is not a type $T_1$ vertex in level $2m+1$, and so neither is $(2m+1)(-g_2)$.
Finally the vertex $-2mg_3-g_2$ (mod $n$) can be reached by a path of length $2m+1$ with edges from two distinct generators. Therefore if it is in level $2m+1$ then it is a type $T_2$ vertex or possibly $T_3$. Here $2n-2mg_3-g_2=64m^3+32m^2+12m+2-2m(16m^2+4m+1)-(4m+1)=32m^3+24m^2+6m+1=(2m+1)(16m^2+4m+1)=(2m+1)g_3$. Thus $(2m+1)g_3$ is not a type $T_1$ vertex in level $2m+1$, and so neither is $(2m+1)(-g_3)$.
Hence there are no type $T_1$ vertices in the first submaximal level, $2m+1$, and consequently none in any later levels. Therefore for isomorphism class 1 the total number of type $T_1$ vertices is $6\times 2m=12m=4k$.

For a largest known circulant graph of degree 6 of isomorphism class 2 and arbitrary diameter $3m$, the order is $DF(6,k)=n=32m^3+16m^2+6m+1$ and a generator set is $\{g_1,g_2,g_3\}$ where $g_1=1$, $g_2=8m^2+2m$ and $g_3=8m^2+6m+2$.
We again consider the six potential type $T_1$ vertices within level $2m+1$: $(2m+1)(\pm g_1)$, $(2m+1)(\pm g_2)$ and $(2m+1)(\pm g_3)$.
First consider the vertex $-g_3-(2m-1)g_2+g_1$ (mod $n$). It can be reached by a path of length $2m+1$ with edges from three distinct generators. Therefore if it is in level $2m+1$ then it is a type $T_3$ vertex. But $n-g_3-(2m-1)g_2+g_1=32m^3+16m^2+6m+1-(8m^2+6m+2)-(2m-1)(8m^2+2m)+1=16m^3+12m^2+2m=(2m+1)(8m^2+2m)=(2m+1)g_2$. Thus $(2m+1)g_2$ is not a type $T_1$ vertex in level $2m+1$, and so neither is $(2m+1)(-g_2)$.
The vertex $-(2m-1)g_3+g_2-g_1$ (mod $n$) can also be reached by a path of length $2m+1$ with edges from three distinct generators. Therefore if it is in level $2m+1$ then it is a type $T_3$ vertex. Now $n-(2m-1)g_3+g_2-g_1=32m^3+16m^2+6m+1-(2m-1)(8m^2+6m+2)+(8m^2+2m)-1=16m^3+20m^2+10m+2=(2m+1)(8m^2+6m+2)=(2m+1)g_3$. Thus $(2m+1)g_3$ is not a type $T_1$ vertex in level $2m+1$, and so neither is $(2m+1)(-g_3)$.
Suppose for a contradiction that the vertex $(2m+1)g_1$ is not a type $T_1$ vertex in level $2m+1$. Then we have $a_1g_1+a_2g_2+a_3g_3\equiv 2m+1$ (mod $n$) with $\vert a_1\vert +\vert a_2\vert +\vert a_3\vert \leq2m+1$ and $\vert a_2\vert +\vert a_3\vert\geq 1$ for some $a_1, a_2, a_3 \in \Z$. Now  $\vert a_1g_1+a_2g_2+a_3g_3\vert \leq(2m+1)g_3=16m^3+20m^2+10m+2<n-(2m+1)$. Thus $a_1g_1+a_2g_2+a_3g_3=2m+1$. We have $g_1\equiv g_3 \equiv 1$ (mod $4m+1$) and $g_2\equiv 0$ (mod $4m+1$). So $a_1+a_3\equiv 2m+1$ (mod $4m+1$). But $\vert a_1\vert +\vert a_2\vert +\vert a_3\vert \leq2m+1$. So $a_1+a_3=2m+1, a_1\geq0$ and $a_3\geq 0$. Hence $a_2=0$. So we have $(2m+1)-a_3+a_3g_3=2m+1$. Hence $a_3(g_3-1)=0$ and $a_3=0$ contradicting the premise. Therefore vertex $(2m+1)g_1$ is a type $T_1$ vertex in level $2m+1$ and consequently so is $(2m+1)(-g_1)$.
Finally we need to prove that the vertex $(2m+2)g_1$ is not a type $T_1$ vertex. The vertex $g_3-g_2-2mg_1$ can be reached by a path of length $2m+2$ with edges from three distinct generators. Therefore if it is in level $2m+2$ then it is a type $T_3$ vertex. Here $g_3-g_2-2mg_1=8m^2+6m+2-(8m^2+2m)-2m=2m+1=(2m+1)g_1$. Thus $(2m+2)g_1$ is not a type $T_1$ vertex in level $2m+2$, and so neither is $(2m+2)(-g_1)$.
Hence there are two type $T_1$ vertices in the first submaximal level, $2m+1$, none in the second, $2m+2$,  and consequently none in any later levels. Therefore for isomorphism class 2 the total number of type $T_1$ vertices is $6\times 2m+2=12m+2=4k+2$.
\end{proof}

\begin{theorem}
For largest known degree $8$ circulant graphs of diameter $k$, the total number of type $T_1$ vertices is equal to $4k+4$.

\label{theorem:7H}
\end{theorem}
\begin{proof}
We give the proof just for the case $k\equiv 0$ (mod 2), $k\geq 4$. The proof for the case $k\equiv 1$ (mod 2) is similar. So let $k=2m$ for $m\geq 2$.
By Theorem \ref{theorem:B3} every distance partition level $l$ of a largest known degree 8 graph of diameter $k$ is maximal for $l\leq\lfloor (k+1)/2\rfloor=m$. Then by Theorem \ref{theorem:7F} the number of type $T_1$ vertices in level $l$ is equal to the degree $d=8$ for $1\leq l \leq m$.

For a largest known circulant graph of degree 8 and arbitrary diameter $2m$, the order is $L(d,k)=n=8m^4+8m^3+12m^2+4m$ and a generator set is $\{g_1,g_2,g_3,g_4\}$ where $g_1=1$, $g_2=4m^3+4m^2+6m+1$, $g_3=4m^4+4m^2-4m$ and $g_4=4m^4+4m^2-2m$. Now we consider the eight potential type $T_1$ vertices within level $m+1$: $(m+1)(\pm g_1)$, $(m+1)(\pm g_2)$, $(m+1)(\pm g_3)$ and $(m+1)(\pm g_4)$, relative to root vertex 0.

First consider the vertex $g_4-g_3-(m-1)g_1$. It can be reached by a path of length $m+1$ with edges from three distinct generators. Therefore if it is in level $m+1$ then it is a type $T_3$ vertex or possibly $T_4$ but not $T_1$. But we have $g_4-g_3-(m-1)g_1=4m^4+4m^2-2m-(4m^4+4m^2-4m)-(m-1)=m+1=(m+1)g_1$. Thus $(m+1)g_1$ is not a type $T_1$ vertex in level $m+1$, and so neither is $(m+1)(-g_1)$. Similarly the vertex $-g_4+g_3-(m-1)g_2$ (mod $n$) can be reached by a path of length $m+1$ with edges from three distinct generators. Therefore if it is in level $m+1$ then it is a type $T_3$ vertex or possibly $T_4$. Now $n-g_4+g_3-(m-1)g_2=8m^4+8m^3+12m^2+4m-(4m^4+4m^2-2m)+(4m^4+4m^2-4m)-(m-1)(4m^3+4m^2+6m+1)=4m^4+8m^3+10m^2+7m+1=(m+1)(4m^3+4m^2+6m+1)=(m+1)g_2$. Thus $(m+1)g_2$ is not a type $T_1$ vertex in level $m+1$, and so neither is $(m+1)(-g_2)$.

Suppose for a contradiction that the vertex $(m+1)g_3$ is not a type $T_1$ vertex in level $m+1$. Then we have $\sum_{i=1}^{4}{a_ig_i}\equiv (m+1)g_3$ (mod $n$) with $\sum_{i=1}^{4}{\vert a_i\vert} \leq m+1$ and $\vert a_1\vert +\vert a_2\vert +\vert a_4\vert\geq 1$ for some $a_i \in \Z$. Now $g_1\equiv g_2\equiv 1$ (mod $2m$) and $g_3 \equiv g_4 \equiv n \equiv 0$ (mod $2m$). Thus $a_1+a_2 \equiv 0$ (mod $2m$), and so $a_2=-a_1$. Also $g_1 \equiv 1, g_2 \equiv 2m+1, g_3 \equiv 0, g_4 \equiv 2m, n \equiv 0$ (all mod $4m$). So $a_1+(2m+1)a_2+2ma_4 \equiv 0$ (mod $4m$), $a_1-a_1(2m+1)+2ma_4 \equiv 0$ (mod $4m$), and $2m(a_4-a_1) \equiv 0$ (mod $4m$). Hence $a_4 \equiv a_1$ (mod 2). Let $a_i=x_im+y_i$ where $x_i, y_i \in \Z$ for $i=1,2,3,4$. Then $x_2=-x_1$ and $y_2=-y_1$. Also $x_4 \equiv x_1$ and $y_4 \equiv y_1$ (both mod 2). As $\sum_{i=1}^{4}{\vert a_i\vert} \leq m+1$, we have $2\vert x_1 \vert+\vert x_3 \vert+\vert x_4 \vert \leq 1$. Thus $x_1=0$. Then $x_4$ is even, and so $x_4=0$. Therefore we have $a_1=y_1, a_2=-y_1, a_3=x_3m+y_3$ and $a_4=y_4$. Hence $\sum_{i=1}^{4}{a_ig_i}-(m+1)g_3=4(x_3-1)m^5+4(y_3+y_4-1)m^4+4(-y_1+x_3-1)m^3+4(-y_1+y_3-x_3+y_4)m^2+2(-3y_1-2y_3-y_4+2)m \equiv 0$ (mod $n$). But $n \equiv 0$ (mod $4m$). So $2m(y_1+2y_4) \equiv 0$ (mod $4m$), and $y_4=-y_1$ (mod 2). As this polynomial is a quintic in $m$ we may set it equal to $(x_5m+y_5)n$ where $x_5, y_5 \in \Z$. Now we equate coefficients of powers of $m$. $m^5$: $x_3-1=2x_5$. $m^4$: $y_3+y_4-1=2x_5+2y_5$. $m^3$: $-y_1+x_3-1=3x_5+2y_5$. $m^2$: $-y_1+y_3-x_3+y_4=x_5+3y_5$. $m$: $-3y_1-2y_3-y_4+2=2y_5$. If $x_3=0$ then this would imply $2x_5=-1$ and so $x_3=\pm 1$. First consider $x_3=1$. Then by simple manipulation we find $x_5=0, y_5=0, y_1=0, y_3=1$ and $y_4=0$. This gives $\sum_{i=1}^{4}{a_ig_i}=(m+1)g_3$, which does not satisfy the premise that we must have $\vert a_1\vert +\vert a_2\vert +\vert a_4\vert\geq 1$. Now consider $x_3=-1$. Then we find $x_5=-1, y_5=0, y_1=1, y_3=0$ and $y_4=-1$. So $\sum_{i=1}^{4}{a_ig_i}=g_1-g_2-mg_3-g_4$, which does not satisfy the premise that $\sum_{i=1}^{4}{\vert a_i\vert} \leq m+1$. Therefore vertex $(m+1)g_3$ is a type $T_1$ vertex in level $m+1$ and consequently so is $(m+1)(-g_3)$.

Similarly, we suppose for a contradiction that the vertex $(m+1)g_4$ is not a type $T_1$ vertex in level $m+1$. Then we have $\sum_{i=1}^{4}{a_ig_i}\equiv (m+1)g_4$ (mod $n$) with $\sum_{i=1}^{4}{\vert a_i\vert} \leq m+1$ and $\vert a_1\vert +\vert a_2\vert +\vert a_3\vert\geq 1$ for some $a_i \in \Z$. Following the same argument as for vertex $(m+1)g_3$ we again find $a_1=y_1, a_2=-y_1, a_3=x_3m+y_3$ and $a_4=y_4$. Hence $\sum_{i=1}^{4}{a_ig_i}-(m+1)g_4=4(x_3-1)m^5+4(y_3+y_4-1)m^4+4(-y_1+x_3-1)m^3+2(-2y_1+2y_3-2x_3+2y_4-1)m^2+2(-3y_1-2y_3-y_4+1)m \equiv 0$ (mod $n$). As this polynomial is a quintic in $m$ we may set it equal to $(x_5m+y_5)n$ where $x_5, y_5 \in \Z$. Now we equate coefficients of powers of $m$. $m^5$: $x_3-1=2x_5$. $m^4$: $y_3+y_4-1=2x_5+2y_5$. $m^3$: $-y_1+x_3-1=3x_5+2y_5$. $m^2$: $-2y_1+2y_3-2x_3+2y_4-1=2x_5+6y_5$. $m$: $-3y_1-2y_3-y_4+1=2y_5$. If $x_3=0$ then this would imply $2x_5=-1$ and so $x_3=\pm 1$. First consider $x_3=1$. Then by simple manipulation we find $2y_5=1$ which is not possible. Now consider $x_3=-1$. This implies $2y_5=-1$ which is also not possible. Therefore vertex $(m+1)g_4$ is a type $T_1$ vertex in level $m+1$ and consequently so is $(m+1)(-g_4)$.

Now we need to prove that the vertex $(m+2)g_3$ is not a type $T_1$ vertex in level $m+2$. The vertex $-g_4-(m-1)g_3-g_2+g_1$ can be reached by a path of length $m+2$ with edges from four distinct generators. Therefore if it is in level $m+2$ then it is a type $T_4$ vertex. We have $mn-g_4-(m-1)g_3-g_2+g_1=8m^5+8m^4+12m^3+4m^2-(4m^4+4m^2-2m)-(m-1)(4m^4+4m^2-4m)-(4m^3+4m^2+6m+1)+1=4m^5+8m^4+4m^3+4m^2-8m=(m+2)(4m^4+4m^2-4m)=(m+2)g_3$. Thus $(m+2)g_3$ is not a type $T_1$ vertex in level $m+2$, and so neither is $(m+2)(-g_3)$. Finally we prove that the vertex $(m+2)g_4$ is not a type $T_1$ vertex in level $m+2$. The vertex $-(m-1)g_4+g_3+g_2-g_1$ (mod $n$) can be reached by a path of length $m+2$ with edges from four distinct generators. Therefore if it is in level $m+2$ then it is a type $T_4$ vertex. Here $(m-1)n-(m-1)g_4+g_3+g_2-g_1=(m-1)(8m^4+8m^3+12m^2+4m)-(m-1)(4m^4+4m^2-2m)+(4m^4+4m^2-4m)+(4m^3+4m^2+6m+1)-1=4m^5+8m^4+4m^3+6m^2-4m=(m+2)(4m^4+4m^2-2m)=(m+2)g_4$. Thus $(m+2)g_4$ is not a type $T_1$ vertex in level $m+2$, and so neither is $(m+2)(-g_4)$.

Hence there are four type $T_1$ vertices in the first submaximal level, $m+1$, none in the second, $m+2$,  and consequently none in any later levels. Therefore the total number of type $T_1$ vertices is $8m+4=4k+4$.
\end{proof}


\section {Conclusion}

We have seen how an analysis of their distance partitions reveals much interesting structure of extremal and largest known circulant graphs up to degree 9. These graphs were all found to have odd girth which is maximal for their diameter. The maximum number of vertices in each level of the distance partition was shown to be related to an established upper bound for the order of Abelian Cayley graphs, $M_{AC}(d,k)$. These graphs all have a maximal zone where the levels achieve this upper bound, and for degree $d\geq 5$ a submaximal zone where they are smaller. Defining the type of each vertex in a level according to the number of adjacent vertices in the preceding level, the number of vertices of each type in each maximal level was also shown to be related to the same upper bound. Finally the total number of type $T_1$ vertices in each of these graphs was determined to be a linear function of their diameter.

We have observed for all the extremal and largest known graphs of degree 4 to 9 that the total number of type $T_1$ vertices increases by 4 for every increase by 1 in the diameter. We have also established that the number of type $T_1$ vertices in each level $l\geq2$ within the maximal zone is twice the dimension $f$, giving $2f$. The resultant ratio of $2/f$ gives a value of 1 for degree 4, 2/3 for degree 6, and 1/2 for degree 8. This correlates with the proportion of levels that lie within the maximal zone for each even degree. We also note for the largest known graphs of degree 6 to 9, having a submaximal zone, that the number of type $T_2$ vertices in each level is initially 0 in level 1, increases by $4{f\choose 2}=2f(f-1)$ per level in the maximal zone, and then decreases at the same rate in the submaximal zone after a limited transition adjustment between the two zones. As the number of such vertices in a level can never be negative, this progression can only exist as long as the submaximal zone does not contain more levels than the maximal zone. This limit is reached at dimension 4, when the maximal zone covers half of the levels. Indeed for the best graph of degree 8 the number of type $T_2$ vertices in level $k$ remains constant at 26 for any diameter $k\geq 3$. Extrapolating these relationships to the case of degree 10, the proportion of levels within the maximal zone would be 2/5 and the number of type $T_2$ vertices in the final level would consequently be negative for any diameter above some threshold, which is of course impossible.

A prime objective of this analysis of properties of the extremal and best graphs of degree 2 to 9 is to discover relationships that may be parametrised by degree and enable extrapolation to degree 10 and beyond. An unfortunate consequence of the failure of the type $T_2$ vertex number calculation for degree 10 is that some key relationships that are valid for circulant graphs up to dimension 4 are found to be invalid for any larger dimension. This means that extremal graphs of degree 10 and above must differ in structure from the solutions for degree 9 and below. Nevertheless it is hoped that elements of the approach presented in this paper will be useful in the search for extremal circulant graphs of higher dimension. It is conjectured that such graphs will continue to have maximal zones where the proportion of levels in the zone depends on the degree but is independent of the diameter. It is also conjectured that the total number of type $T_1$ vertices will be a linear function of the diameter with coefficient 4, independent of the degree. However it is quite possible that the graphs will not have odd girth that is maximal, $2k+1$.

The question also arises whether such an analysis of distance partitions and vertex types might provide the basis for a proof that the largest known circulant graphs of degree 6 to 9 are extremal for all diameters above their known thresholds.



\end{document}